\documentclass{amsart}
\usepackage{verbatim,amssymb,amsmath,amscd,latexsym,amsbsy,mathrsfs} 
\usepackage{graphicx}
\input{xy}
\xyoption{all}

\newtheorem{thm}{Theorem}[section]
\newtheorem{cor}[thm]{Corollary}

\DeclareMathOperator*{\im}{im} 
 
\DeclareMathOperator*{\rank}{rank} 
\DeclareMathOperator*{\id}{id}

\newcommand{\QQ}{\mathbb{Q}}
\newcommand{\ZZ}{\mathbb{Z}}

\newcommand{\PP}{\mathbb{P}}

\newcommand {\C} {{\mathbb C}}
\newcommand {\R} {{\mathbb R}}
\newcommand {\Z} {{\mathbb Z}}
\newcommand {\Q} {{\mathbb Q}}

\newcommand {\E} {{\mathcal E}}
\newcommand {\dt} {{\bullet}}

 \newtheorem{lemma}{Lemma}[section]
 \newtheorem{prop}{Proposition}[section]
 \newtheorem{remark}{Remark}[section]

\begin{document}
\title{ A new class of surfaces with maximal Picard number}
\author{
        Donu Arapura    
}
\thanks{First author partially supported by NSF}
% \address{Department of Mathematics\\
% Purdue University\\
% West Lafayette, IN 47907\\
% U.S.A.}
\author{Partha Solapurkar
}
\address{Department of Mathematics\\
Purdue University\\
West Lafayette, IN 47907\\
U.S.A.}
 \maketitle

 \begin{abstract}
   A new class of examples of surfaces with maximal Picard number is
   constructed. These carry pencils of genus two or three curves such
   their Jacobian fibrations are isogenous to fibre products of
   elliptic modular surfaces.
 \end{abstract}

Let us say that  a smooth complex projective surface is Picard maximal
if the Picard number $\rho$ equals the Hodge number $h^{11}$; in
other words, if $\rho$  is as large as possible. It is easy to see that any surface with trivial geometric
genus is Picard maximal, but finding examples with $p_g>0$ is much
subtler. Picard maximal abelian or K3 surfaces can be constructed by
taking a product of a CM elliptic curve with itself or the associated
Kummer surface. Shioda \cite{shioda} showed that
all elliptic modular surfaces are Picard maximal.
The first published examples of general type go back to Persson
\cite{persson}, who used double covers branched over rather special
configurations of curves. A few more
examples have since been found, and we refer to the  Beauville's article
\cite{beauville} for a survey. The goal of this note is
give some new, and we believe rather natural, examples, most of which have general type. The
inspiration for us came from  Shioda's work \cite{shioda}
mentioned above. Our examples, which
carry pencils of genus two or three curves, are related in the sense
that the Jacobian fibrations
are isogenous to  fibre products of elliptic modular
surfaces. This is the key point that makes the examples work.
 It is worth remarking that  these examples are defined over $\bar \Q$, and they give new
examples where Tate's conjecture holds.

\section{Hodge theory of fibered surfaces}

Let $X$ be a smooth complex projective surface. Our interest is in $H^2(X)$. This
carries a canonical weight $2$ Hodge structure with respect to which the
Neron-Severi group $NS(X)$ can be identified with $H^2(X,\Z)\cap
H^{11}(X)$ by Lefschetz's theorem. This yields the well known
inequality $\rho(X)= \rank NS(X)\le h^{11}(X)$. We will say that $X$ is
{\em Picard maximal}, if equality  holds. Part of the interest in this
class stems from the following observation of Faltings  \cite[pp
81-82]{tate}: If $X$ is defined over a
finitely generated subfield $k\subset \C$ and $X_\C$ is Picard
maximal, then Tate's conjecture holds for $X$, i.e. Galois invariant
part of \'etale cohomology $H_{et}^2(X_{\bar k},\Q_\ell(1))^{Gal(\bar k/k)}$ is
generated by divisors.

Now suppose that $X$ carries a surjective
holomorphic map $f:X\to C$ with connected fibres  to a  smooth
projective curve. We suppose also that $f$ possesses a section
$\sigma:C\to X$.
Let $g$ be the genus of the general fibres of $f$
and let $q$ be the genus of $C$.  Also let  $j:U\to C$ be a nonempty
Zariski open
set over which $X^o=f^{-1}U\to U$ is smooth.
We can analyze $H^2(X)$ using either the
Leray spectral sequence or the decomposition theorem. We use the
latter since it  is bit more convenient. By Saito's version of the
decomposition theorem \cite[p 857]{saito}, we can decompose $\R f_*\Q$ 
as a sum of intersection cohomology complexes up to shift in the
constructible derived category, and moreover these complexes underly
pure Hodge modules. 
By restricting this sum to $U$ and
applying Deligne's theorem \cite{deligne-L},   we can identify some of 
these components explicitly:
\begin{equation}
  \label{eq:decomp1}
\R f_*\Q \cong  \underbrace{\Q}_{f_*\Q}\oplus  j_*j^*R^1f_*\Q[-1]\oplus  \underbrace{\Q}_{j_*j^*R^2f_*\Q}[-2]\oplus
  M
\end{equation}
The, as yet undetermined, term $M$  is supported on the finite set $S=C-U$.
This yields a (noncanonical) decomposition
\begin{equation}
  \label{eq:decomp2}
H^{2}(X,\Q) \cong   f^*H^2(C,\Q)\oplus IH^1(R^1f_*\Q)\oplus H^0(C, j_*j^*R^2f_*\Q)\oplus H^2(M)  
\end{equation}
where $IH^1(R^1f_*\Q)=H^1(j_*j^*R^1f_*\Q)$. The first summand on the
right is spanned by the fundamental class $[X_t]$ of a fibre. The
third summand is spanned by the class $[\sigma]$.
To calculate $M$, we restrict to a point $s\in S$, and observe that $H^*(\R
f_*\Q|_s)$ is the cohomology of the fibre $X_s=f^{-1}(s)$ by proper
base change \cite[p 41]{dimca}. Therefore $M$
gives the excess cohomology not coming from the  the preceding terms in \eqref{eq:decomp1}. Let
$$X_s=\sum_{i=1}^{m_s} n_{s,i}X_{s,i}$$
be the decomposition into irreducible components. Let $D_s$ be a small disk
centered at $s$, $t\in D_s^*=D_s-\{s\}$, and $\gamma_s\in
\pi_1(D_s^*,t)$ a generator. Then after
combining the local invariant cycle theorem \cite{schmid} with some
elementary topological arguments, we see that
$$H^i(X_s,\Q)=
\begin{cases}
  \Q  &\text{if $i=0$}\\
H^1(X_t,\Q)^{\gamma_s} \cong (j_*j^*R^1f_*\Q)_s&\text{if $i=1$}\\
  \Q^{m_s}  &\text{if $i=2$}\\
  0 &\text{otherwise}
\end{cases}
$$
Therefore $M \cong \bigoplus \Q_s^{m_s-1}[-2]$.
From \eqref{eq:decomp2}, we deduce that we have
a noncanonical decomposition
\begin{equation}
  \label{eq:decomp}
H^{2}(X) \cong   IH^1(R^1f_*\Q)\oplus \Q[\sigma]\oplus\Q[X_t]  \oplus \bigoplus_{s} \Q^{m_s-1}  
\end{equation}
We can see that the last two summands are spanned by divisor classes
supported on the fibres. Since the decomposition \eqref{eq:decomp1}
can be lifted to the derived category of Hodge modules, we see that \eqref{eq:decomp} becomes
an isomorphism of Hodge structures provided that all the summands in \eqref{eq:decomp}  except the
first are viewed as sums of  the  Tate structures $\Q(-1)$. The first
summand  $IH^1(R^1f_*\Q)$  is the interesting piece. We note that
an intrinsic  Hodge structure on  it  was first constructed by
Zucker \cite{zucker}.
Since all  but the
first summand on the right of \eqref{eq:decomp} are spanned by divisor
classes, we may conclude that:

\begin{prop}
  If $f:X\to C$ is a map satisfying the above assumptions, then $X$ is
  Picard maximal if and only if $IH^1(R^1f_*\C)^{(1,1)}$ is spanned by
  divisors. In particular,  this is the case if $IH^1(R^1f_*\C)^{(1,1)}=0$. 
\end{prop}

Although the decomposition \eqref{eq:decomp} is not canonical, we note
that $IH^1(R^1f_*\Q)$ is the canonical subquotient $L^1/L^2$ of
$H^2(X)$, where $L^\dt$ is filtration associated to the Leray spectral
sequence. In particular, given a commutative diagram
$$
\xymatrix{
 X'\ar[r]^{g}\ar[rd]^{f'} & X\ar[d]^{f} \\ 
  & C
}
$$
the map $g^*:H^2(X)\to H^2(X')$ will take $IH^1(R^1f_*\Q)$ to
$IH^1(R^1f'_*\Q)$. 

We recall that the Mordell-Weil group $MW(X/C)$ is the group of
sections of the associated Jacobian fibration $Pic^0(X^o/U)$. The group is finitely
generated if $Pic^0(X^o/U)$  has no fixed
  part,  i.e. no nonzero constant abelian subschemes \cite{conrad}.
We will say that a surface satisfying $IH^1(R^1f_*\C)^{(1,1)}=0$ is
{\em extremal}. While this  terminology is not very descriptive,  it conforms to
standard usage in elliptic surface theory \cite[p 75]{miranda} because of the following:

\begin{lemma}
  Suppose that the Jacobian fibration associated to $X^o/U$ has no fixed
  part.  If $X/C$ is extremal then the rank of $MW(X/C)$ is zero. The converse holds when $X$ is Picard maximal.
\end{lemma}

\begin{proof}
  Formula \eqref{eq:decomp} implies
$$\dim IH^1(R^1f_*\C)^{(1,1)} = h^{11}(X) - 2-\sum (m_s-1)$$
The lemma is consequence of this together with the Shioda-Tate formula
\cite[(4)]{shioda2}
$$\rank MW(X/C) = \rho(X) - 2-\sum (m_s-1)$$
\end{proof}

Given a finite index subgroup $\Gamma\subseteq SL_2(\Z)$, we have an
associated modular curve $U_\Gamma$ given as a quotient of the upper half
plane by $\Gamma$. Let $C_\Gamma$ be the nonsingular compactification.
We can interpret $U_\Gamma$ as the  moduli space of elliptic curves with
some sort of level structure. The associated elliptic modular surface
$\mathcal{E}_\Gamma\to C_\Gamma$ is the universal family of elliptic curves
over $U_\Gamma$ (minus the set of elliptic fixed points) suitably extended to $C_\Gamma$. As Shioda observed
\cite[7.8]{shioda} such surfaces are Picard maximal. In fact, the
stronger property also holds.

\begin{thm}[Shioda]\label{thm:shioda} Suppose that $-I\notin \Gamma$
  and let  $f:\mathcal{E}_\Gamma\to C_\Gamma$ be  corresponding
  elliptic modular surface. Then $\mathcal{E}_\Gamma\to C_{\Gamma}$ is extremal.
\end{thm}

\begin{proof}
  This is proved in \cite[4.12]{shioda}. (Shioda formulates this using
  Kodaira's homological invariant $G$, but this can be identified with $j_*j^*R^1f_*\Z$.)
\end{proof}

\begin{remark}
  The conclusion also applies  to the Legendre family
  $y^2=x(x-1)(x-\lambda)$, which is the elliptic modular surface for
  $\Gamma(2)$, even though the theorem does not. The point is this
  surface is rational and therefore Picard maximal. Furthermore, Igusa \cite{igusa}
  has shown that the rank of the Mordell-Weil group is zero.
\end{remark}

\section{Frey-Kani construction}

Given an integer $n>0$, by degree $n$ FK data, we will mean the following: a pair of elliptic curves $E, E'$, and an
isomorphism of $n$-torsion subgroups $\phi:E[n]\to E'[n]$ which is an
anti-isometry with respect to the Weil pairings. Let
$\Gamma_\phi\subset (E\times E')[n]$ be subgroup given by the graph
of $\phi$. 

\begin{thm}[Frey-Kani]\label{thm:fk} Given degree $n$  FK data,
  the abelian surface $J=E\times E'/\Gamma_\phi$ carries a unique
  principal polarization $\Theta$ such that $\pi^*\Theta \sim n( E\times
  0+0\times E')$ where $\pi:E\times E'\to J$ is the projection. Either
  \begin{enumerate}
  \item $\Theta$ is an irreducible smooth curve of genus $2$, or
\item $\Theta$ is a sum of two elliptic curves meeting at one point.
  \end{enumerate}
The second case can only happen if $E$ and $E'$ are isogenous.
\end{thm}

\begin{proof}
This was proved in   \cite{fk} with some addition restrictions. These
were relaxed in  \cite[1.5, 2.3]{kani}.
\end{proof}

If case (1) holds above, we will say that the FK data is {\em irreducible}, or that the anti-isometry $ \phi $ is {\em irreducible}.

\begin{thm}[Kani]
If $ n $ is prime, then 
\begin{enumerate}
\item There exists an irreducible anti-isometry $\phi:E[n]\to E'[n]$ for any pair of elliptic curves $E,E'$.
\item An anti-isometry $ \phi: E[n] \to E'[n] $ is reducible if and only if there is an isogeny $ h: E \to E' $ of degree $ k(n-k) $ for some $ 1 \leq k < n $, such that 
\[ \phi \circ [k] = h\vert_{E[n]}. \]

\end{enumerate}
\end{thm}

\begin{proof}
  This was proved in \cite[Theorems 2 and 3]{kani}. 
\end{proof} 

\begin{cor}
  Let $ E $ be an elliptic curve without complex multiplication and let $ n $ be a prime such that $ n \equiv 3 \mod 4 $. Then all the anti-isometries $ \phi: E[n] \to E'[n] $ are irreducible. 
\end{cor}

\begin{proof}
  We prove that there are no isogenies $ h: E \to E $ of degree $ k(n-k) $ for any $ 1 \leq k < n $. If $ h:E \to E $ is an isogeny, then $ \deg(h) = m^2 $ for some integer $ m $, since $ E $ does not have complex multiplication. But $ m^2 $ cannot equal $ k(n-k) $, since otherwise $ -1 $ would be a quadratic residue modulo $ n $. 
\end{proof}

\begin{cor}
  Let $ E $ be any elliptic curve such that $ j(E) \neq 0, 1728 $. Then all the anti-isometries $ \phi: E[2] \to E[2] $ which are not equal to the identity are irreducible. 
\end{cor}

\begin{proof}
  For $ n=2 $, anti-isometries are exactly the same as isometries. Since $ E $ has no automorphisms other than $ \pm \id_E $ when $ j(E) \neq 0, 1728 $, all the anti-isometries which are not the identity are irreducible. 
\end{proof} 

\begin{remark}\label{rmk:irred}
  The $\Theta$-divisors associated to  $\phi$ and $-\phi$ are the
same as abstract curves. Therefore $\phi$ is irreducible if and only
if $-\phi$ is. Thus we may extend this terminology to orbits $\{\pm \phi\}$.
\end{remark}

\section{The examples}

Let  $G=SL_2(\Z/n\Z)$ and let $\Gamma(n) = \ker [SL_2(\Z)\to G]$ be the
principal congruence group of level $n$. Let $f:\E=\E_{\Gamma(n)}\to C_{\Gamma(n)}=C$  be  the associated elliptic
modular surface, and let $U\subset C$ be the complement of the
discriminant. The group $G$ acts on $C$ through the quotient $\bar
G =PSL_2(\Z/n\Z)= Gal(C/C_{\Gamma(1)})$. Let $f_\sigma:\E_\sigma\to
C$ be the pullback of  $\E\to C$ along $\sigma:C\to C$. This only
depends on the image $\bar \sigma= \im \sigma\in \bar G$, so we may
also
denote this by $f_{\bar \sigma} :\E_{\bar \sigma}\to C$.
Fix $t\in U$ and a reference anti-isometry
 $\phi_1:(\Z/n\Z)^2\to (\Z/n\Z)^2$  represented by say $\begin{pmatrix}
  0 & 1\\ 1 & 0
\end{pmatrix}$.
 Observe that the fibres $E=\E_t$ and $E'= \E_{\sigma,t}$ are the 
same elliptic curve equipped with different level $n$ structures
$(\Z/n\Z)^2\cong E[n]$ and $(\Z/n\Z)^2\cong E'[n]$. Thus $\phi_1$
induces an anti-isometry $\phi_\sigma:E[n]\cong E'[n]$.

We need to make the construction of $\phi_\sigma$ a bit more precise.
Let $\tilde f:\tilde \E\to H$ be the universal marked elliptic
curve over the upper half plane. By ``marked'', we mean that there is
fixed symplectic isomorphism $\lambda:R^1\tilde f_{*}\Z\cong \Z^2$,
where the right side is equipped with the standard pairing represented
by the matrix 
$
\begin{pmatrix}
  0 & 1\\ -1 & 0
\end{pmatrix}
$.
The group $SL_2(\Z)$ acts equivariantly  on $\tilde f$, and $R^1\tilde
f_*\Z$
is the equivariant constant sheaf associated to the standard
representation of this group. More concretely, this means that for
$\tilde \sigma\in SL_2(\Z)$, the base change map $\tilde \sigma^*R^1\tilde f_*\Z\to
R^1\tilde f_*\Z$ corresponds, under $\lambda$, to multiplication by $\tilde
\sigma$ on $\Z^2$. 
Given a preimage  ${\tilde\sigma}\in SL_2(\Z)$ of $\sigma\in G$, we have
a commutative diagram
$$
\xymatrix{
  & \tilde \E\ar[dd]^{\tilde f}\ar[ld]\ar[rr]^{\tilde \sigma} &  & \tilde \E\ar[dd]^{\tilde f}\ar[ld] \\ 
 \E_\sigma\ar[dd]^{f_\sigma}\ar[rr] &  & \E\ar[dd]^{f} &  \\ 
  & H\ar[rr]^{\tilde\sigma}\ar[ld]^{\pi} &  & H\ar[ld]^{\pi} \\ 
 U\ar[rr]^{\sigma} &  & U & 
}
$$
where for simplicity of notation we have omitted the restriction symbols ``$|_{f^{-1}U}$''.
With a little bit of thought, one sees that all the vertical squares are Cartesian.
We can deduce from this, and the fact that $U=H/\Gamma(n)$, that $\lambda$ modulo $n$
descends to  isomorphisms $\lambda_\sigma:R^1f_{\sigma  *}\Z/n\Z|_{U}\cong
(\Z/n\Z)^2$ for each $\sigma$. Furthermore, we get a commutative diagram
$$
\xymatrix{ (\Z/n\Z)^2\ar^{\sigma}[d] &\pi^*R^1f_*\Z/n\Z|_U
  \ar_{\lambda_1}[l]\ar@{-->}^{\Lambda_{\sigma}}[d]\\
  (\Z/n\Z)^2\ar^{\lambda_{\sigma}^{-1}}[r] & \pi^*R^1f_{\sigma *} \Z/n\Z|_U}
$$
of constant local systems which can be descended to $U$.
Composing $\Lambda_{\sigma}$ with the reference anti-isometry $\phi_1$ gives a new
anti-isometry $\phi_{\sigma}:R^1f_*\Z/n\Z|_U\to R^1f_{\sigma
  *}\Z/n\Z|_U$. This can be viewed as an anti-isometry  from $\E[n]|_U\to
\E_\sigma[n]|_U$, thanks to the canonical  isomorphisms $\E[n]|_U\cong
R^1f_*\Z/n\Z|_U$ and $\E'[n]|_U\cong \linebreak R^1f_{\sigma *}\Z/n\Z|_U$.

The   set  of anti-isometries of $(\Z/n\Z)^2\to
(\Z/n\Z)^2$ forms a torsor over $G$. In other
words, all anti-isometries   are  given by composing $\phi_1$ with an element of $G$.
Therefore as $\sigma$ varies in $ G$
we obtain all possible anti-isometries $\E_t[n]\to
\E_{\sigma(t)}[n]=\E_t[n]$.
We say that $\sigma\in G$ is irreducible if the  FK data $(E=\E_t, E'
= \E_{\sigma(t)},\phi_{\sigma})$ is irreducible.
By the results of the previous section, all $\sigma \in G$ are irreducible when $ n \equiv 3 \mod 4 $ is a prime. For $ n=2 $, all five $ \sigma \in SL_2(\ZZ/2\ZZ) $ such that $ \sigma \neq \id_2 $ are irreducible. Moreover,  when $ n \equiv 1 \mod 4$ is a prime, there is at least one irreducible $ \sigma \in G $.

Fix an irreducible element $\sigma\in G=SL_2(\Z/n\Z)$.  Let
$X_{n,\sigma}^o\subset \E\times_U\E_\sigma/\Gamma_\phi$ be the relative $\Theta$
divisor with respect to a principal polarization satisfying the
conditions of theorem \ref{thm:fk} on the fibres. By assumption, some fibre of $X_{n,\sigma}^o\to C$ is a smooth
irreducible curve of genus $2$, therefore the same holds for 
all fibres over some Zariski open $V\subset U$ containing $t$. Let $F:X_{n,\sigma}\to C$ be a relatively minimal nonsingular compactification
of  the preimage of $V$ in $X_{n,\sigma}^o$.

\begin{thm}
With the notation as in the last paragraph,  $F:X_\sigma\to C$ is extremal, and therefore Picard maximal.
\end{thm}

\begin{proof} Let us replace $U$ by $V$.
  Since $Pic^0(X_{n,\sigma}^o/U)\to \E\times_U\E_\sigma$ is a fibrewise
  isogeny, $R^1F_*\Q|_U \cong (R^1f_*\Q \oplus  R^1f_{\sigma *}\Q)|_U$
  as Hodge modules. Note also that $ R^1f_{\sigma*}\Q|_U\cong  \sigma^*R^1f_{*}\Q|_U$. Therefore
 $$IH^1(R^1F_*\C)^{(1,1)}=  IH^1(R^1f_*\C)^{(1,1)} \oplus IH^1( R^1f_{*}\C)^{(1,1)}=0$$
\end{proof}

We  mention a  few related examples.
\begin{enumerate}
\item[(A)] Choose a finite index subgroup $\Gamma\subseteq
\Gamma(n)$. Let $X_{\Gamma,\sigma}$ be a minimal desingularization of
the fibre product
$X_{n,\sigma}\times_{C_{\Gamma(n)}}C_{\Gamma}$. By the  same argument, we
can see that $X_{\Gamma, \sigma}\to C_{\Gamma}$ is extremal, and
consequently Picard maximal.
\smallskip

\item[(B)]  

Let $f:X\to \PP^1$ be obtained by blowing up $\PP^2$ along the base locus of the pencil
$$\{x^4+y^4+z^4 + t(x^2y^2+y^2z^2+z^2x^2)=0\}_{t\in \PP^1}$$
These curves are nonsingular when $t\notin S=\{-1, \pm 2,\infty\}$. 
 Beauville \cite[p 5]{beauville} shows  that when $t\notin S$, the
 curve $E_t=X_t/\{x\mapsto \pm x\}$  can
be identified with the elliptic curve $ y^2 =x(x-1)(x+t+1)$, and
moreover the map induces an isogeny $Pic^0(X_t)\sim E_t^3$. 
 In fact this is an isogeny of abelian schemes  from
$Pic^0(X/\PP^1-S)$ to $\E^3|_{\{\PP^1-S\}}$, where $g:\E\to C_{\Gamma(2)}=\PP^1$ is the Legendre
family  pulled backed along the automorphism that fixes $ 0,1 $ and sends $t \mapsto -t-1$.
Thus 
\begin{equation}
  \label{eq:Rfg}
  R^1f_*\Q\cong (R^1g_*\Q)^3
\end{equation}
 and therefore $X\to \PP^1$ is
extremal. It follows that $X$ is  Picard maximal but this was already clear from the fact
that it was rational. However, we can 
create a  nonrational surface by  choosing a subgroup
$\Gamma\subseteq \Gamma(2)$ of sufficiently large finite index, and letting
 $X\to C_\Gamma$ be the desingularized pullback of $S$  to
$C_{\Gamma}$ under the projection $C_\Gamma\to C_{\Gamma(2)}=\PP^1$.
The isomorphism \eqref{eq:Rfg}
will persist if we pull it back to $C_\Gamma$. Thus we see that $X$ is
also extremal and hence Picard maximal.
\smallskip 
\item[(C)]

Let $ E_t $ be the elliptic curve given by \[ y^2 = - \frac{t^3}{(1+t)^3}(x-1)(x+t)(x+1/t),\] for $ t \in U = \PP^1 \backslash \{ 0, \pm 1, \infty \} $. Then $ E_t $ is isomorphic to the elliptic curve given by $ y^2 = x(x-1)(x-t) $ as follows: The elliptic curve $ y^2 = x(x-1)(x-t) $ is the double cover of $ \PP^1 $ branched over $ \{t, 1, 0, \infty\} $. If one pulls back this elliptic curve along the automorphism of $ \PP^1 $ that sends \[ x \mapsto \frac{-t}{1+t}(x-1), \] one gets $ E_t $ branched over $ \{-t, -1/t, 1, \infty\} $. This works in families, and the family $ E_t $ is isomorphic to the Legendre family. 

Now consider the family $ C_t $ of genus $ 2 $ curves given by \[ y^2 =-\frac{t^3}{(1+t)^3} (x^2-1)(x^2+t)(x^2+1/t).\] These curves are nonsingular for $ t \in U $. We have a surjective map $ C_t \to E_t $ which sends $ (x,y) \mapsto (x^2,y) $. Under this map, the differential $ dx/2y \in H^0(E_t, \omega_{E_t}) $ pulls back to $ xdx/y \in H^0(C_t,\omega_{C_t}) $. Consider the automorphism $ \tau $ of $ C_t $ given by $ x \mapsto 1/x $ and $ y \mapsto y\sqrt{-1}/x^3 $. We have $ \tau^*(xdx/y) = \sqrt{-1}dx/y \in H^0(C_t,\omega_{C_t}) $. Thus the two differential forms $ xdx/y $ and $ \tau^*(xdx/y) $ give a basis for $ H^0(C_t,\omega_{C_t}) $. Now the proof of Lemma 2 in \cite[p 4]{beauville} gives an isogeny $ Pic^0(C_t) \sim E_t^2 $. 

Let $ \mathcal C $ denote the family of the curves $ C_t $ as $ t $ varies in $ U $, with the associated map $ f : \mathcal C \to U $, and let $ \mathcal E $ denote the family of the curves $ E_t $ as $ t $ varies in $ U $ with the associated map $ g : \mathcal E \to U $. The argument above works in a family and gives an isogeny of abelian schemes from $ Pic^0(\mathcal C/U) $ to $ \mathcal E \times_U \mathcal E $ for all $ t \in U $. Then \[ R^1f_*\QQ = (R^1g_*\QQ)^{\oplus 2}, \] and hence $ \mathcal C \to U $ is extremal. Thus, $ \mathcal C $ is Picard maximal. 

\end{enumerate}


\begin{thebibliography}{99}



\bibitem[B]{beauville} A. Beauville, {\em Some surfaces with maximal
    Picard number}, 	arXiv:1310.3402



\bibitem[C]{conrad} B. Conrad, {\em Chow's K/k-image and K/k-trace,
    and the Lang-N\'eron theorem.} Enseign. Math. (2) 52 (2006),
37 - 108.



\bibitem[D1]{deligne-L} P. Deligne, {\em Th\'eor\'eme de Lefschetz et
    critres de degenerescence de suites spectrales} Publ. IHES  35 (1968): 259-78.



\bibitem[Di]{dimca} A. Dimca, {\em Sheaves in topology}, Springer (2004)


\bibitem[FK]{fk} G. Frey, E. Kani, {\em Curves of genus 2 covering elliptic curves and an arithmetical application} in: “Arithmetic Algebraic Geometry” (G. van der Geer, F. Oort, J. Steenbrink, eds.)





\bibitem[I]{igusa} J. Igusa, {\em Fibre systems of Jacobian varietes
    III }, Amer. J. Math 81 (1959)




\bibitem[K]{kani} E. Kani, {\em The number of curve of genus two with
    elliptic differentials},   J. Reine Angew. Math. 485 (1997)




\bibitem[M]{miranda} R. Miranda, {\em The basic theory of elliptic surfaces.} ETS Editrice, Pisa, (1989)



\bibitem[P]{persson} U. Persson, {\em Horikawa surfaces with maximal
    Picard numbers}, Math Ann. 259 (1982)



\bibitem[Sa]{saito} M. Saito {\em Module Hodge polarizable},
  Publ. Res. Inst. Math. Sci. 24 (1988)



\bibitem[Sc]{schmid} W. Schmid, {\em Variation of Hodge structure},
  Invent. Math. 22 (1972)



\bibitem[Sh]{shioda} T. Shioda, {\em On elliptic modular surfaces},
  J. Math. Soc. Jap. 24  (1971)

\bibitem[Sh2]{shioda2} T. Shioda, {\em Mordell-Weil lattices for
    higher genus fibration}, Proc. Japan Acad 68 (1992)



\bibitem[T]{tate} J. Tate, {\em Conjectures on algebraic cycles in
    l-adic cohomology.} Motives (Seattle, WA, 1991), 71 - 83, AMS (1994)

\bibitem[Z]{zucker} S. Zucker, {\em Hodge theory  with degenerating
    coefficients}, Annals Math 109  (1979)
\end{thebibliography}
\end{document}